\documentclass{article}

\usepackage[top=3cm,bottom=3.5cm,inner=3cm,outer=3cm]{geometry}
\usepackage{amsthm,amsmath,amssymb,mathtools}
\usepackage{enumerate, enumitem, csquotes,verbatim}
\usepackage{setspace}
\usepackage[numbers]{natbib}
\usepackage{hyperref}
\usepackage{color}

\MakeOuterQuote{"}

\makeatletter
\def\imod#1{\allowbreak\mkern10mu({\operator@font mod}\,\,#1)}

\makeatother

\theoremstyle{plain}
\newtheorem{theorem}{Theorem}
\newtheorem{lemma}[theorem]{Lemma}

\newtheorem{proposition}[theorem]{Proposition}

\newtheorem{corollary}[theorem]{Corollary}

\newtheorem{conjecture}[theorem]{Conjecture}
\newtheorem{problem}[theorem]{Problem}

\newtheorem{case}{Case}
\newtheorem{subcase}{Case}
\numberwithin{subcase}{case}
\newtheorem{subsubcase}{Case}
\numberwithin{subsubcase}{subcase}

\newcommand{\mycase}[1]{\begin{case}#1\end{case}}
\newcommand{\mysubcase}[1]{\begin{subcase}#1\end{subcase}}

\title{Partite Saturation of Complete Graphs}

\author{Ant\'onio Gir\~ao\thanks{\,Department of Pure Mathematics and Mathematical Statistics, University of Cambridge, Wilberforce Road, Cambridge CB3 0WB, UK;
    \texttt{A.Girao@dpmms.cam.ac.uk}.}
   \and Teeradej Kittipassorn\thanks{\,Departamento de Matem\'atica, Pontif\'icia Universidade Cat\'olica do Rio de Janeiro (PUC-Rio), Rua Marqu\^es de S\~ao Vicente 225, G\'avea, Rio de Janeiro, RJ 22451-900, Brazil; \texttt{ping41@gmail.com}.}
  \and Kamil Popielarz\thanks{\,Department of Mathematical Sciences, University of Memphis,
      Memphis, TN 38152, USA; \texttt{kamil.popielarz@\-gmail.com} }}

\begin{document}
\maketitle

\begin{abstract}
We study the problem of determining $sat(n,k,r)$, the minimum number of edges in a $k$-partite graph $G$ with $n$ vertices in each part such that $G$ is $K_r$-free but the addition of an edge joining any two non-adjacent vertices from different parts creates a $K_r$. Improving recent results of Ferrara, Jacobson, Pfender and Wenger, and generalizing a recent result of Roberts, we define a function $\alpha(k,r)$ such that $sat(n,k,r) = \alpha(k,r)n + o(n)$ as $n \rightarrow \infty$. Moreover, we prove that
	\[
		k(2r-4) \le \alpha(k,r) \le
			\begin{cases}
				(k-1)(4r-k-6) &\text{ for }r \le k \le 2r-3,\\
				(k-1)(2r-3) &\text{ for }k \ge 2r-3,
			\end{cases}
	\]
and show that the lower bound is tight for infinitely many values of $r$ and every $k\geq 2r-1$. 
This allows us to prove that, for these values, $sat(n,k,r) = k(2r-4)n + O(1)$ as $n \rightarrow \infty$.
Along the way, we disprove a conjecture and answer a question of the first set of authors mentioned above. 
\end{abstract}


\section{Introduction}

Given a graph $H$, the classical Tur\'an-type extremal problem asks for the maximum number of edges in an $H$-free graph on $n$ vertices. While the corresponding minimization problem is trivial, it is interesting to determine the minimum number of edges in a maximal $H$-free graph on $n$ vertices. We say that a graph is \emph{$H$-saturated} if it is $H$-free but the addition of an edge joining any two non-adjacent vertices creates a copy of $H$. The minimum number $sat(n,H)$ of edges in an $H$-saturated graph on $n$ vertices was first studied in 1949 by Zykov~\citep{Zykov} and independently in 1964 by Erd{\H o}s, Hajnal, and Moon~\citep{Erdos} who proved that $sat(n,K_r) = (r-2)(n-1) - \binom{r-2}{2}$. Soon after this, Bollob\'as~\citep{Bollobashyper} determined exactly $sat(n, K_r^{(s)})$ where $K_r^{(s)}$ is the complete $s$-uniform hypergraph on $r$ vertices. Later, in 1986, K\'aszonyi and Tuza~\citep{Kaszonyi} showed that the saturation number $sat(n,H)$ for a graph $H$ on $r$ vertices is maximized at $H = K_r$, and consequently, $sat(n,H)$ is linear in $n$ for any $H$. For results on the saturation number, we refer the reader to the survey~\citep{Faudree}.

This concept can be generalized to the notion of $H$-saturated subgraphs which are maximal elements of a family of $H$-free subgraphs of a fixed host graph. A subgraph of a graph $G$ is said to be \emph{$H$-saturated in $G$} if it is $H$-free but the addition of an edge in $E(G)$ joining any two non-adjacent vertices creates a copy of $H$. The problem of determining the minimum number $sat(G,H)$ of edges in an $H$-saturated subgraph of $G$ was first proposed in the above mentioned paper of Erd{\H o}s, Hajnal, and Moon. They conjectured a value for the saturation number $sat(K_{m,n},K_{r,r})$ which was verified independently by Bollob\'as~\citep{Bollobas1,Bollobas2} and Wessel~\citep{Wessel1,Wessel2}. Very recently, Sullivan and Wenger~\citep{Wenger} studied the analogous saturation numbers for tripartite graphs within tripartite graphs and determined $sat(K_{n_1,n_2,n_3},K_{l,l,l})$ for every fixed $l\geq 1$ and every $n_1,n_2$ and $n_3$ sufficiently large. Several other host graphs have been considered, including hypercubes~\citep{Choi,Johnson,Morrison} and random graphs~\citep{Korandi}. 

In this paper, we are interested in the saturation number $sat(n,k,r)=sat(K_{k \times n},K_r)$ for $k \ge r \ge 3$ where $K_{k \times n}$ is the complete $k$-partite graph containing $n$ vertices in each of its $k$ parts. This function was first studied recently by Ferrara, Jacobson, Pfender and Wenger~\citep{Ferrara} who determined $sat(n,k,3)$ for $n\ge 100$.  Later, Roberts~\citep{Roberts} showed that $sat(n,4,4)=18n-21$ for sufficiently large $n$.

For convenience, we say that a $k$-partite graph with a fixed $k$-partition is \emph{$K_r$-partite-saturated} if it is $K_r$-free but the addition of an edge joining any two non-adjacent vertices from different parts creates a $K_r$. Therefore, $sat(n,k,r)$ is the minimum number of edges in a $k$-partite graph $G$ with $n$ vertices in each part which is $K_r$-partite-saturated. 

Our first result states that $sat(n,k,r)$ is linear in $n$ where the constant $\alpha(k,r)$ in front of $n$ is defined as follows.  Given $k \ge r \ge 3$, consider a $K_r$-partite-saturated $k$-partite graph $G$ containing an independent set $X$ of size $k$ consisting of exactly one vertex from each part of $G$. We define $\alpha(k,r)$ to be the minimum number of edges between $X$ and $X^c$ taken over all such $G$ and $X$.

\begin{theorem}\label{thm:sat}
For $k \ge r \ge 3$,
	\[
		sat(n,k,r) = \alpha(k,r)n + o(n)
	\]
as $n \rightarrow \infty$. 
\end{theorem}

Let us shift our focus to the function $\alpha(k,r)$. The next theorem states what we know about it.

\begin{theorem}\label{thm:alpha}
For $k \ge r \ge 3$,
    \begin{enumerate}[label={$(\roman*)$}, ref={\thelemma$(\roman*)$}]
		\item $k(2r-4) \le \alpha(k,r) \le
			\begin{cases}
				(k-1)(4r-k-6) &\text{ for }r \le k \le 2r-3,\\
				(k-1)(2r-3) &\text{ for }k \ge 2r-3.
			\end{cases}$
        \item $\alpha(k,r) = k(2r-4)$ if
			$\begin{cases}
				k = 2r-3,\text{ or}\\
				k \ge 2r-2 \text{ and }r \equiv 0 \mod 2,\text{ or}\\
                k \ge 2r-1 \text{ and }r \equiv 2 \mod 3.
			\end{cases}$
        \item $\alpha(k,3) = 3(k-1)$, $\alpha(4,4) = 18$ and $33 \le \alpha(5,5) \le 36$.
		\item $\alpha(r,r) \ge r(2r-4)+1$ for $r \ge 4$.
	\end{enumerate}
\end{theorem}

The bounds in $(i)$, together with Theorem~\ref{thm:sat}, imply that $sat(n,k,r) = O(krn)$, answering a question of Ferrara, Jacobson, Pfender and Wenger~\citep{Ferrara}. In $(ii)$, we determine exactly $\alpha(k,r)$ for some values of $r$ and every $k$ large enough, allowing us to disprove a conjecture in \citep{Ferrara} which states that $sat(n,k,r) = (k-1)(2r-3)n - (2r-3)(r-1)$ for $k \ge 2r-3$ and sufficiently large $n$. In $(iii)$, we deal with the cases $r = 3,4,5$ which have not been covered by $(ii)$. Finally, $(iv)$ shows that  the lower bound in $(i)$, which is attained for certain values of $r$ and $k$ mentioned in $(ii)$, is not tight when $k=r$.

Theorem~\ref{thm:sat} and Theorem~\ref{thm:alpha} imply that $sat(n,k,r) = k(2r-4)n + o(n)$ for the values of $k$ and $r$ in $(ii)$. We show that, in this case, the $o(n)$ term can be replaced by a constant.

\begin{theorem}\label{thm:constant}
For $k \ge r \ge 3$,
	\[
		sat(n,k,r) = k(2r-4)n + O(1)\text{ if }
		\begin{cases}
				k = 2r-3,\text{ or}\\
				k \ge 2r-2 \text{ and }r \equiv 0 \mod 2,\text{ or}\\
				k \ge 2r-1 \text{ and }r \equiv 2 \mod 3,
		\end{cases}
	\]
as $n \rightarrow \infty$.
\end{theorem}

Now we give a summary of the values of $sat(n,k,r)$ in the case $r=3,4,5$ which are immediate consequences of the first three results.

\begin{corollary}
        \begin{enumerate}[label={$(\roman*)$}, ref={\thelemma$(\roman*)$}]
		\item $sat(n,k,3) = 3(k-1)n + o(n)$ for $k \ge 3$ and as $n \rightarrow \infty$.
		\item $sat(n,k,4) =
			\begin{cases}
				18n + o(n) &\text{ for }k=4,\text{ as }n \rightarrow \infty,\\
				4kn + O(1) &\text{ for }k \ge 5,\text{ as }n \rightarrow \infty.
			\end{cases}$
		\item $sat(n,k,5)
			\begin{cases}
				\in \left[33n + o(n),36n + o(n)\right] \qquad\text{ for }k=5,\text{ as }n \rightarrow \infty,\\
				\in \left[36n + o(n),40n + o(n)\right] \qquad\text{ for }k=6,\text{ as }n \rightarrow \infty,\\
				\in \left[48n + o(n),49n + o(n)\right] \qquad\text{ for }k=8,\text{ as }n \rightarrow \infty,\\
				= 6kn + O(1) \quad\text{ for } k=7 \text{ or } k \ge 9,\text{ as }n \rightarrow \infty.\qed
			\end{cases}$
	\end{enumerate}
\end{corollary}

We note that $(i)$ and the first half of $(ii)$ are not the best known results. In fact, Ferrara, Jacobson, Pfender and Wenger~\citep{Ferrara} proved that $sat(n,k,3) = 3(k-1)n-6$ for sufficiently large $n$ and Roberts~\citep{Roberts} proved that $sat(n,4,4) = 18n-21$ for sufficiently large $n$.

Let us give some more definitions which will be used throughout the paper. For a $k$-partite $G=V_1\cup V_2\cup\dots\cup V_k$, we refer to each $V_i$ as a \emph{part of $G$}. We say that an edge (or a non-edge) $uv$ of a $k$-partite graph is \emph{admissible} if $u,v$ lie in different parts. We say that a non-edge $uv$ of a $K_r$-free graph is \emph{$K_r$-saturated} if adding $uv$ to the graph completes a $K_r$. In other words, a $k$-partite graph is $K_r$-partite-saturated if it is $K_r$-free and every admissible non-edge is $K_r$-saturated.

The rest of this paper is organized as follows. Section~\ref{sec:thm:sat} is devoted to the proof of Theorem~\ref{thm:sat}. In Section~\ref{sec:alpha}, we study the function $\alpha(k,r)$ and prove Theorem~\ref{thm:alpha}$(i)$. In Section~\ref{sec:alpha(ii)}, we prove Theorem~\ref{thm:alpha}$(ii)$ by describing constructions matching the lower bound $\alpha(k,r) \ge k(2r-4)$ in Theorem~\ref{thm:alpha}$(i)$. We prove Theorem~\ref{thm:alpha}$(iii)$, Theorem~\ref{thm:alpha}$(iv)$ and Theorem~\ref{thm:constant} in Section~\ref{sec:alpha(iii)}, Section~\ref{sec:diagonal} and Section~\ref{sec:thm:constant} respectively. Finally, we conclude the paper in Section~\ref{sec:conclude} with some open problems.


\section{Proof of Theorem~\ref{thm:sat}}
\label{sec:thm:sat}

First we show that the upper bound follows easily from the definition of $\alpha(k,r)$.

\begin{proposition}\label{prop:satupper}
For every $k \ge r \ge 3$ and any integer $n\geq \alpha(k,r)+1$, we have $sat(n,k,r) \le \alpha(k,r)n + \alpha(k,r)^2$.
\end{proposition}
\begin{proof}
Let $G$ be a $K_r$-partite-saturated $k$-partite graph containing an independent set $X$ of size $k$ consisting of exactly one vertex from each part of $G$ with $e(X,X^c)=\alpha(k,r)$. We may assume that $|X^c| \le \alpha(k,r)$. Indeed, since there are $\alpha(k,r)$ edges between $X$ and $X^c$, deleting all the vertices in $X^c$ with no neighbors in $X$ leaves at most $\alpha(k,r)$ vertices in $X^c$. Note that any admissible non-edge with at least one endpoint in $X$ is still $K_r$-saturated. We finish by keeping adding admissible edges inside $X^c$ until every admissible non-edge inside $X^c$ is $K_r$-saturated.

Let $V_1,V_2,\dots,V_k$ be the parts of $G$. It follows that $|V_i| = |V_i \cap X| + |V_i \cap X^c| \le 1 + \alpha(k,r) \le n$, and so we can modify $G$ to have exactly $n$ vertices in each part by blowing up the vertex of $X$ in $V_i$ to a class of size $n - |V_i \cap X^c|$ for each $i$. The resulting graph is $K_r$-partite-saturated and has exactly $n$ vertices in each of its $k$ parts. Moreover, the number of edges is at most $\alpha(k,r)n + e(G[X^c]) \le \alpha(k,r)n + \alpha(k,r)^2$.
\end{proof}

Now we prove the lower bound $sat(n,k,r) \ge \alpha(k,r)n + o(n)$.

Let $\varepsilon > 0$ and let $G = V_1 \cup V_2 \cup \dots \cup V_k$ be a $K_r$-partite-saturated $k$-partite graph with $|V_i|=n$ for all $i \in [k]$. We shall show that $e(G) \ge \alpha(k,r)n - \varepsilon n$ for all sufficiently large $n$. Let $d$ be a large natural number to be chosen later. For each $i$, we partition $V_i$ into $V_i^+=\{v \in V_i: d(x) \ge d\}$ and $V_i^-=\{v \in V_i: d(x) < d\}$. First we show that $V_i^+$ is small. Since $e(G) \ge \frac{d}{2}|V_i^+|$, we are done unless $|V_i^+| \le \frac{2\alpha(k,r)}{d}n$. Now we show that we can delete a constant number of vertices from $\bigcup_{i=1}^{k} V_i^-$ to make it independent.

\begin{lemma}\label{lemma:lower}
There exists a subset $U \subset \bigcup_{i=1}^{k} V_i^-$ of size $C_{k,d}$ such that $\left(\bigcup_{i=1}^{k} V_i^-\right)\setminus U$ forms an independent set in $G$ for some constant $C_{k,d}$. 
\end{lemma}

Let us first show how to finish the proof of Proposition~\ref{prop:satupper} using the lemma. For each $1 \le i \le k$, let $v_i$ be a vertex of smallest degree in $V_i^-\setminus U$. Since $G$ is a $K_r$-partite-saturated $k$-partite graph and $X = \{v_1,v_2,\dots,v_k\}$ is an independent set with exactly one vertex in each part of $G$, we have $\sum_{i=1}^{k} d(v_i)\ge \alpha(k,r)$ by the definition of $\alpha(k,r)$. Since $\left(\bigcup_{i=1}^{k} V_i^-\right)\setminus U$ forms an independent set,
	\begin{align*}
		e(G)
		&\ge \sum_{i=1}^{k} \sum_{v\in V_i^-\setminus U} d(v)
		\ge \sum_{i=1}^{k} |V_i^-\setminus U| d(v_i)
		\ge (n - |V_i^+| - |U|)\sum_{i=1}^{k} d(v_i) \\
		&\ge \alpha(k,r)\left(n - \frac{2\alpha(k,r)}{d}n - C_{k,d}\right)
		= \alpha(k,r)n - \left(\frac{2\alpha(k,r)^2}{d} + \frac{\alpha(k,r)C_{k,d}}{n}\right)n
		\ge \alpha(k,r)n - \varepsilon n
	\end{align*}
by taking $d$ and $n$ sufficiently large. It remains to prove the lemma.

\begin{proof}[Proof of Lemma~\ref{lemma:lower}]
It is sufficient to show that any matching between $V_i^-$ and $V_j^-$ has size less than $4^{d^2}$ for all $i\not=j$. Indeed, we can take $U$ to be the endpoints of maximal matchings between $V_i^-$ and $V_j^-$ for all $i\not=j$ and $|U| < 4^{d^2}\binom{k}{2}$.

Suppose for contradiction that $\{x_1 y_1,x_2 y_2,\dots,x_{4^{d^2}} y_{4^{d^2}}\}$ is a matching of size $4^{d^2}$ where $X=\{x_1,x_2,\allowbreak\dots, x_{4^{d^2}}\}\subset V_1^-$ and $Y=\{y_1,y_2,\dots, y_{4^{d^2}}\}\subset V_2^-$. The strategy of the proof is to iteratively find vertices $x_{t_1},x_{t_2},\dots,x_{t_{d}}$ of $X$ such that $d(x_{t_{i}}) \ge i$ for all $1\le i\le d$, which would contradict the fact that $x_{t_{d}} \in V_1^-$. In fact, we shall find vertices $x_{t_1},x_{t_2},\dots,x_{t_{d}}$ of $X$ such that
	\begin{itemize}
        \item[(i)]  there exists a common neighbor of $x_{t_i}$ and $y_{t_j}$ which is not a neighbor of $y_{t_1},y_{t_2},\dots,y_{t_{j-1}}$ for all $i > j$.
	\end{itemize}
Clearly, this implies that $d(x_{t_{i}}) \ge i$ for all $1\le i\le d$. To find such vertices, it is sufficient to find vertices $x_{t_1},x_{t_2},\dots,x_{t_{d}}$ of $X$ satisfying
	\begin{itemize}
        \item[(ii)]  $x_{t_i}$ and $y_{t_j}$ are not neighbors for all $i>j$, and
        \item[(iii)] $N(x_{t_{i}})\cap N(y_{t_{l}})=N(x_{t_{j}})\cap N(y_{t_{l}})$ for all $i>j>l$.
	\end{itemize}
First we show that (ii) and (iii) imply (i). Let $i>j$. By (ii), $x_{t_i}y_{t_j}$ is a non-edge. Since $G$ is $K_r$-partite-saturated, there exists a clique $W$ of size $r-2$ in the common neighborhood of $x_{t_i}$ and $y_{t_j}$. Since $r \ge 3$, we are done by picking a required vertex from $W$ unless each vertex in $W$ is joined to some $y_{t_l}$ with $l<j$. In this case, $W \cup \{x_{t_j},y_{t_j}\}$ forms a clique of size $r$, contradicting the fact that $G$ is $K_r$-free. Indeed, each $w\in W$ belongs to some $N(y_{t_l})$ with $l<j$, and since $w \in N(x_{t_i})$, we must have $w \in N(x_{t_j})$, by (iii).

Now, we find vertices $x_{t_1},x_{t_2},\dots, x_{t_{d}}$ of $X$ satisfying (ii) and (iii). To help us do so, we shall iteratively construct a nested sequence of sets $X\supset X_1\supset X_2\supset \dots\supset X_d$ with $x_{t_i}\in X_i$ for all $2\leq i \leq d$, satisfying
	\begin{itemize}
		\item[(iv)]  $x$ and $y_{t_{i-1}}$ are not neighbors for all $x\in X_i$, and
		\item[(v)] $N(x)\cap N(y_{t_{i-1}})=N(x')\cap N(y_{t_{i-1}})$ for all $x,x'\in X_i$.
	\end{itemize}
Clearly, such vertices $x_{t_1},x_{t_2},\dots,x_{t_{d}}$ satisfy (ii) and (iii). Start with $x_{t_1}=x_1$ and $X_1=X$. Let $i\le d$ and suppose that we have found vertices $x_{t_1},x_{t_2},\dots,x_{t_{i-1}}$ and sets $X_1\supset X_2\supset \dots\supset X_{i-1}$ with $x_{t_j}\in X_j$ for all $j<i$, satisfying (iv) and (v). We delete the neighbors of $y_{t_{i-1}}$ from $X_{i-1}$ and partition the remaining vertices into $2^{d(y_{t_{i-1}})} \le 2^{d}$ subsets according to their common neighborhood with $y_{t_{i-1}}$. In other words, $X_{i-1} \setminus N(y_{t_{i-1}})$ is partitioned into subsets $\{x : N(x)\cap N(y_{t_{i-1}})=S\}$ for $S\subset N(y_{t_{i-1}})$. We choose $X_i$ to be such subset of maximum size, i.e. $|X_i| \ge \frac{|X_{i-1}| - d}{2^d}$. Clearly, $X_i$ satisfies (iv) and (v). We then choose $x_{t_i}$ be any vertex in $X_i$. It remains to prove that $|X_i|>0$. Recall that $|X_1|=|X|=4^{d^2}$, and we can see, by induction, that $|X_i| \ge 4^{d(d-i)}$ for $i\le d$. Indeed,
	\[
		|X_i| \ge \frac{|X_{i-1}| - d}{2^d} \ge \frac{|X_{i-1}|}{4^d} \ge \frac{4^{d(d-i+1)}}{4^d} \ge 4^{d(d-i)}
	\]
as required.
\end{proof}

\section{Bounding \texorpdfstring{$\alpha(k,r)$}{alpha(k,r)}}
\label{sec:alpha}

In this section, we establish a number of results that will help us prove Theorem~\ref{thm:alpha}. We shall deduce Theorem~\ref{thm:alpha}$(i)$ at the end of the section.

For $k \ge r \ge 2$ and $1 \le i \le k-r+1$, let $\beta_i(k,r)$ be the minimum number of vertices in a $K_r$-free $k$-partite graph such that the subgraph induced by any $k-i$ parts contains a $K_{r-1}$, i.e. the deletion of any $i$ parts does not destroy all the $K_{r-1}$.

We observe that $\beta_1$ and $\beta_2$ are useful for bounding $\alpha$.

\begin{proposition}\label{prop:alpha}
For $k \ge r \ge 3$,
	\[
		k\beta_1(k-1,r-1) \le \alpha(k,r) \le (k-1)\beta_2(k,r-1).
	\]
\end{proposition}
\begin{proof}
To prove the lower bound, let $G$ be a $K_r$-partite-saturated $k$-partite graph containing an independent set $X$ of size $k$ consisting of exactly one vertex from each part of $G$. We shall show that $e(X,X^c) \ge k\beta_1(k-1,r-1)$. It is sufficient to show that each vertex in $X$ has degree at least $\beta_1(k-1,r-1)$. Let $x\in X$ and consider the $(k-1)$-partite graph $H=G[N(x)]$. Clearly, it is $K_{r-1}$-free since $G$ is $K_{r}$-free. It remains to show that, for each part $U$ of $H$, $H\setminus U$ contains a $K_{r-2}$. If $x'$ is a vertex of $X$ in the corresponding part of $U$ in $G$ then, since the non-edge $xx'$ is $K_r$-saturated in $G$, $H\setminus U$ must contain a $K_{r-2}$. Hence, $|N(x)| = |H| \ge \beta_1(k-1,r-1)$.

For the upper bound, let $G_1$ be a $K_{r-1}$-free $k$-partite graph on $\beta_2(k,r-1)$ vertices such that the subgraph induced by any $k-2$ parts contains a $K_{r-2}$. Let $G_2$ be the graph obtained from $G_1$ by adding one vertex of $X=\{x_1,x_2,\dots,x_k\}$ to each part of $G_1$ and joining each $x_i$ to every vertex of $G_1$ outside its part. By construction, $X$ forms an independent set and $e(X,X^c) = (k-1)\beta_2(k,r-1)$ edges. Note that $G_2$ is $K_r$-free since a clique in $G_2$ contains at most one vertex from $X$ and $G_1$ is $K_{r-1}$-free. Now, let $G$ be the graph obtained from $G_2$ by adding admissible edges inside $X^c$, until every admissible non-edge inside $X^c$ is $K_r$-saturated. To conclude that $G$ is $K_r$-partite-saturated, we need to show that every admissible non-edge inside $X$ is $K_r$-saturated. Note that, for every pair of distinct vertices $x,x'\in X$, $G_1$ contains a $K_{r-2}$ not using vertices from the parts containing $x$ and $x'$. Since $x$ and $x'$ are joined to every vertex outside their parts, the addition of the edge $xx'$ completes a $K_r$. Hence, $\alpha(k,r) \le e(X,X^c) = (k-1)\beta_2(k,r-1)$.
\end{proof}

In the next sections, the argument above used in the proof of the lower bound will be used several times. Let us state it as a lemma.
\begin{lemma}\label{lem:deg}
Let $G$ be a $k$-partite $K_r$-free graph containing an independent set $X$ of size $k$ consisting of exactly one vertex from each part of $G$ such that the non-edges inside $X$ are $K_r$-saturated. Then, for each $x \in X$, $G[N(x)]$ is a $K_{r-1}$-free $(k-1)$-partite graph such that the subgraph induced by any $k-2$ parts contains a $K_{r-2}$. In particular, $d(x) \ge \beta_1(k-1,r-1)$ for all $x \in X$.
\end{lemma}

In the next two subsections, we shall bound $\beta_1$ from below and $\beta_2$ from above.

\subsection{Upper bounds for \texorpdfstring{$\beta_i$}{beta i}}

We start with an easy observation which helps us bound $\beta_i$ from above.

\begin{lemma}\label{lem:betaupper}
 For $k \ge r \ge 3$ and $1 \le i \le k-r+1$, $\beta_i(k,r) \le \beta_i(k-1,r-1) + i + 1$.
\end{lemma}
\begin{proof}
Let $H=U_1\cup U_2\cup \dots\cup U_{k-1}$ be a $K_{r-1}$-free $(k-1)$-partite graph on $\beta_i(k-1,r-1)$ vertices such that the subgraph induced by any $k-i-1$ parts contains a $K_{r-2}$. We shall construct a $K_r$-free $k$-partite graph $G=V_1\cup V_2\cup \dots\cup V_k$ from $H$ with $|G|=|H|+(i+1)$ as follows. First, add new vertices $v_1$ to $U_1$, $v_2$ to $U_2$, $\dots$, $v_i$ to $U_i$ and $v_{i+1}$ to the new part $V_k$. This is possible since $k \ge i+2$. Now, join $v_{i+1}$ to every vertex in $H$ and, for every $1 \le j \le i$, join $v_j$ to every vertex in $H\setminus U_j$. Clearly, $G$ is $K_r$-free since $H$ is $K_{r-1}$-free.

Let $\mathcal{C}$ be a collection of $k-i$ parts of $G$. It remains to check that the subgraph of $G$ induced by $\mathcal{C}$ contains a $K_{r-1}$. First, suppose that $V_k \in \mathcal{C}$. By the induction hypothesis, the other $(k-1)-i$ parts $\mathcal{C}\setminus\{V_k\}$ induce a subgraph of $H$ containing a $K_{r-2}$. Together with $v_{i+1}\in V_k$, they form a $K_{r-1}$ in the subgraph of $G$ induced by $\mathcal{C}$ as required. Now, let us suppose that $V_k \not\in \mathcal{C}$. Then $\mathcal{C}$ must contain at least one of $V_1,V_2,\dots,V_i$. Without loss of generality, we may assume that $\mathcal{C}$ contains $V_1$. By the induction hypothesis, the other $(k-1)-i$ parts $\mathcal{C}\setminus\{V_1\}$ induce a subgraph of $H$ containing a $K_{r-2}$. Together with $v_1\in V_1$, they form a $K_{r-1}$ in the subgraph of $G$ induced by $\mathcal{C}$ as required.
\end{proof}

Lemma~\ref{lem:betaupper} immediately implies the following upper bound on $\beta_i$.

\begin{corollary}\label{cor:betaupper}
$\beta_i(k,r) \le (i+1)(r-1)$ for $k \ge r \ge 2$ and $1 \le i \le k-r+1$.
\end{corollary}
\begin{proof}
It is clear that $\beta_i(k,2)=i+1$ for $k \ge i+1$ by considering the empty graph on $i+1$ vertices where each vertex is in a different part and the remaining $k-i-1$ parts are empty.

By induction on $r$ and applying Lemma~\ref{lem:betaupper}, $\beta_i(k,r) \le \beta_i(k-1,r-1)+i+1 \le (i+1)(r-2)+i+1 = (i+1)(r-1)$ as required.
\end{proof}

We remark that there is a straightforward construction proving Corollary~\ref{cor:betaupper} for the case $k \ge (i+1)(r-1)$, namely, a disjoint union of $i+1$ cliques of size $r-1$ where each vertex is in a different part and the remaining $k-(i+1)(r-1)$ parts are empty. Clearly, the deletion of any $i$ parts does not destroy all the $K_{r-1}$.

Now we prove a better upper bound for $\beta_i(k,r)$ in the case when $i \ge 2$ and $k \ge i(r-1)+1$ by considering the $(r-2)$th power of the cycle $C_{i(r-1)+1}$.

\begin{proposition}\label{prop:betaupper}
$\beta_i(k,r) \le i(r-1)+1$ for $k \ge i(r-1)+1$ and $r,i \ge 2$.
\end{proposition}
\begin{proof}
Since $\beta_i(k,r)$ is decreasing in $k$ (by adding empty parts), it is enough to show that $\beta_i(k,r) \le i(r-1)+1$ for $k=i(r-1)+1$. Let $G$ be the $(r-2)$th power of the cycle $C_{i(r-1)+1}$, i.e. $G$ is a graph on $\mathbb{Z}_{i(r-1)+1}$ where $u,v$ are neighbors if $u-v = 1,2,\dots,r-2$. We view $G$ as a $(i(r-1)+1)$-partite graph with one vertex in each part. Clearly, $G$ is $K_r$-free if $i \ge 2$. Note that, after deleting any $i$ vertices of $G$, there are at least $r-1$ consecutive vertices remaining in $\mathbb{Z}_{i(r-1)+1}$, which form a $K_{r-1}$ as required.
\end{proof}

Proposition~\ref{prop:betaupper} together with Lemma~\ref{lem:betaupper} imply a better upper bound than that in Corollary~\ref{cor:betaupper} for $\beta_2(k,r)$ in the remaining cases, i.e when $k < 2r-1$.

\begin{proposition}\label{prop:beta2upper}
$\beta_2(k,r) \le	4r-k-2$ for $2 \le r < k < 2r-1$. 
\end{proposition}
\begin{proof}
We proceed by induction on $2r-k$. The base case when $2r-k = 1$ follows from Proposition~\ref{prop:betaupper}. Now, suppose that $2r-k \ge 2$. Applying Lemma~\ref{lem:betaupper},
	\[
		\beta_2(k,r) \le \beta_2(k-1,r-1)+3 \le (4(r-1)-(k-1)-2)+3 = 4r-k-2,
	\]
by the induction hypothesis, since $2r-k > 2(r-1)-(k-1) \ge 1$,
\end{proof}

Let us remark that a similar upper bound for general $\beta_i$ can be obtained by the same method. We believe that the bound in Proposition~\ref{prop:beta2upper} is, in fact, an equality.

\begin{conjecture}\label{conj:beta2}
$\beta_2(k,r) =	4r-k-2$ for $2 \le r < k < 2r-1$. 
\end{conjecture}
For the remaining values of $k$, we shall see in the next subsection that $\beta_2(k,r) = 2r-1$ for $k \ge 2r-1$.

\subsection{Determining \texorpdfstring{$\beta_1$}{beta 1}}

We shall show that the upper bound for $\beta_1$ given by Corollary~\ref{cor:betaupper} is an equality. Recall that the clique number of a graph is the order of a maximum clique. 

\begin{proposition}\label{prop:beta1}
$\beta_1(k,r) = 2(r-1)$ for $k \ge r \ge 2$.
\end{proposition}

The lower bound, is a consequence of the following observation.

\begin{proposition}\label{prop:intersection}
Let $G$ be a graph on at most $2s-1$ vertices with clique number $s$. Then there is a vertex which lies in every $K_s$ of $G$.
\end{proposition}

\begin{proof}[Proof of Proposition~\ref{prop:beta1}]
The upper bound follows from Corollary~\ref{cor:betaupper}. To prove the lower bound, suppose for contradiction that $G$ is a $K_r$-free $k$-partite graph on at most $2r-3$ vertices such that the subgraph induced by any $k-1$ parts contains a $K_{r-1}$. Applying Proposition~\ref{prop:intersection} with $s=r-1$, there is a vertex $v$ which lies in every $K_{r-1}$. In particular, the deletion of the part containing $v$ destroys all the $K_{r-1}$. Hence, $\beta_1(k,r) \ge 2r-2$.
\end{proof}

Let us remark that Proposition~\ref{prop:intersection} is a consequence of the clique collection lemma of Hajnal~\citep{Hajnal} which states that the sum of the number of vertices in the union and the intersection of a collection of maximum cliques is at least twice the clique number. Our argument below can also be used to give a new proof of Hajnal's clique collection lemma.

\begin{proof}[Proof of Proposition~\ref{prop:intersection}]
Let $V_1,V_2,\dots,V_m \subset V(G)$ be the vertex sets of the copies of $K_s$ in $G$. For a vertex $v\in V(G)$, let $I_v = \{i\in[m] : v\in V_i\}$ be the set of $K_s$ containing $v$. For a collection $\mathcal{C}\subset\mathcal{P}([m])$ of subsets of $[m]$, let $V_{\mathcal{C}} = \{v\in V(G) : I_v\in\mathcal{C}\}$. Observe that if $\mathcal{C}\subset\mathcal{P}([m])$ is intersecting then $V_{\mathcal{C}}$ induces a clique in $G$. Indeed, $u,v \in V_{\mathcal{C}}$ are neighbors since $I_u \cap I_v \not= \emptyset$, i.e. there is a clique containing both $u$ and $v$. Therefore, $|V_{\mathcal{C}}|\le s$ since $G$ is $K_{s+1}$-free. The following lemma implies the result.

\begin{lemma}
For $m\ge 3$, there exist intersecting families $\mathcal{C}_1,\mathcal{C}_2,\dots \mathcal{C}_{m-2}\subset\mathcal{P}([m])$ such that, for $I\subset [m]$, the number of $\mathcal{C}_j$ containing $I$ is
	$\begin{cases}
    0	& \text{if } I=\emptyset\\
		|I|-1	& \text{if } I\not=\emptyset,[m]\\
    m-2 & \text{if } I=[m].
	\end{cases}$
\end{lemma}
	
\begin{proof}
The proof is by induction on $m$. For $m=3$, $\mathcal{C}_1=\{ \{1,2\}, \{2,3\}, \{3,1\}, \{1,2,3\} \}$ satisfies the required property. For $m\ge 4$, suppose by induction that there exist intersecting families $\mathcal{C}_1,\mathcal{C}_2,\dots,\mathcal{C}_{m-3}\subset\mathcal{P}([m-1])$ satisfying the property. We define $\mathcal{D}_1,\mathcal{D}_2,\dots,\mathcal{D}_{m-2}\subset\mathcal{P}([m])$ as follows. For $1\le j\le m-3$, let 
		\[
			\mathcal{D}_j = \mathcal{C}_j \cup \{I\cup\{m\} : I\in\mathcal{C}_j\}
		\]
	and
		\[
			\mathcal{D}_{m-2} = \{I\subset [m]: m\in I \text{ and } |I|\ge 2\} \cup \{[m-1]\}.
		\]
It is easy to check that $\mathcal{D}_1,\mathcal{D}_2,\dots,\mathcal{D}_{m-2}$ satisfy the required property.
	\end{proof}

Let us deduce the result. This is trivial when $m=1,2$ so we may assume that $m\ge 3$. Observe that
		\[
			\sum_{i=1}^m{|V_i|} = \left|\bigcup_{i=1}^m{V_i}\right| + \sum_{j=1}^{m-2}{|V_{\mathcal{C}_j}|} + \left|\bigcap_{i=1}^m{V_i}\right|.
		\]
Indeed, a vertex $v$ is counted on both sides $|I_v|$ times by the lemma. Using $|V_i| = s$, $\left|\bigcup_{i=1}^m{V_i}\right|\le 2s-1$ and $|V_{\mathcal{C}_j}|\le s$, we have
		\[
			ms \le (2s-1) + (m-2)s + \left|\bigcap_{i=1}^m{V_i}\right|
		\]
i.e.
		$
			\left|\bigcap_{i=1}^m{V_i}\right| \ge 1
		$
as required.
\end{proof}

We remark that the fact that $\beta_1(k,r) = 2(r-1)$ allows us to show that the upper bound for $\beta_2(k,r)$ when $k \ge 2r-1$ in Proposition~\ref{prop:betaupper} is an equality.

\begin{corollary}\label{cor:beta2}
$\beta_2(k,r) = 2r-1$ for $k \ge 2r-1$ and $r \ge 2$.
\end{corollary}
\begin{proof}
Observe that $\beta_i(k,r) \ge \beta_{i-1}(k-1,r)+1$. Indeed, if $G$ is a $K_r$-free $k$-partite graph on $\beta_i(k,r)$ vertices such that the subgraph induced by any $k-i$ parts contains a $K_{r-1}$, then, by deleting a non-empty part of $G$, we obtain a $K_r$-free $(k-1)$-partite graph such that the subgraph induced by any $(k-1)-(i-1)$ parts contains a $K_{r-1}$. This graph must contains at least $\beta_{i-1}(k-1,r)$ vertices and therefore, $|G|-1 \ge \beta_{i-1}(k-1,r)$.

Hence, $\beta_2(k,r) \ge \beta_1(k-1,r)+1 = 2(r-1)+1 = 2r-1$ by Proposition~\ref{prop:beta1}.
\end{proof}

\subsection{Proof of Theorem~\ref{thm:alpha}\texorpdfstring{$(i)$}{(i)}}

The lower bound follows from Proposition~\ref{prop:alpha} and Proposition~\ref{prop:beta1}. The upper bound follows from Proposition~\ref{prop:alpha}, Proposition~\ref{prop:beta2upper} and Corollary~\ref{cor:beta2}.
\qed

\section{Proof of Theorem~\ref{thm:alpha}\texorpdfstring{$(ii)$}{(ii)}}
\label{sec:alpha(ii)}

For $k=2r-3$, we are done since the lower and upper bounds in Theorem~\ref{thm:alpha}$(i)$ match, i.e. $\alpha(k,r) = k(2r-4) = (k-1)(2k-3)$.

Now we shall describe constructions that match the lower bound $\alpha(k,r) \ge k(2r-4)$ in Theorem~\ref{thm:alpha}$(i)$ for the cases when ($k \ge 2r-2$  and $r$ is even) and ($k \ge 2r-1$  and $r=2\mod 3$), i.e. a $K_r$-partite-saturated $k$-partite graph $G$ containing an independent set $X$ of size $k$ consisting of exactly one vertex from each part of $G$ with $e(X,X^c) = k(2r-4)$. Lemma~\ref{lem:deg} tells us that such graph must satisfy $d(x)=2r-4$, for all $x \in X$.

Note that we do not have to worry about making the admissible non-edges inside $X^c$, $K_r$-saturated since we can keep adding admissible edges inside $X^c$ until every admissible non-edge inside $X^c$ is $K_r$-saturated.

Let $p \in \{2,3\}$ be a divisor of $r-2$. First we shall construct such $k$-partite graph $G$, for $k=2r-4+p$. We define $X = \{x_1,x_2,\dots,x_k\}$ and $X^c = \{y_1,y_2,\dots,y_k\}$, where the parts of $G$ are $\{x_i,y_i\}$, for $i=1,2,\dots,k$. There are no edges inside $X$. Let $y_iy_j$ be an edge iff $i,j$ are not consecutive elements of the circle $\mathbb{Z}_k$, and so $G[X^c]$ is the graph $K_k$ minus a cycle $C_k$. Let $x_iy_j$ is an edge iff $i \not= j \mod \frac{k}{p}$, i.e. $x_i$ is joined to all but $p$ equally spaced $y_j$. We claim that $G$ satisfies the required properties.

Clearly, we have $d(x) = k-p = 2r-4$ for all $x \in X$ and $e(X,X^c) = k(2r-4)$. Let us verify that $G$ is $K_r$-free. A clique inside $X^c$ is a set of non-consecutive elements of $\mathbb{Z}_k$, and so a largest clique inside $X^c$ has size $\left\lfloor\frac{k}{2}\right\rfloor = r-1$ for $p \in \{2,3\}$. Since a clique which is not inside $X^c$ can contain at most one vertex of $X$, it remains to check that the neighborhood of each $x_i$ does not contain a clique of size $r-1$. Viewing $X^c$ as a circle, $N(x_i)$ consists of $p$ segments of the circle, each of size $\frac{2r-4}{p}$, separated by gaps of size one. Since $\frac{2r-4}{p}$ is even, a largest clique in $N(x_i)$ has size $\frac{p(2r-4)}{2p} = r-2$.

It remains to show that the admissible non-edges inside $X$, and those between $X$ and $X^c$ are $K_r$-saturated. Let $x_iy_j$ be an admissible non-edge, and so $j=i\pm\frac{k}{p}$ in $\mathbb{Z}_k$. Clearly, $N(x_i)$ contains $r-2$ vertices which form a non-consecutive set of the circle with $y_j$. Therefore, there exists a $K_{r-2}$ in the common neighborhood of $x_i$ and $y_j$ as required. Now let $x_ix_j$ be an admissible non-edge. Then the common neighborhood of $x_i$ and $x_j$ consists of $2p$ segments of the circle separated by gaps of size one such that they form $p$ pairs where the sum of the sizes of each pair is $\frac{2r-4}{p}-1$, and so each pair consists of a segment of even size and a segment of odd size. Therefore, a largest non-consecutive set in $N(x_i)\cap N(x_j)$ has size $\frac{p(2r-4)}{2p} = r-2$. Hence, there exists a $K_{r-2}$ in $N(x_i)\cap N(x_j)$ as required.

We have constructed such $k$-partite graph $G_k$ for $k=2r-4+p$ with . Let us obtain $G_k$ for $k>2r-4+p$ from $G_{2r-4+p}$ by blowing up $x_1$ to a class $\{x_1\} \cup \{x_i:2r-3+p \le i \le k\}$ of size $k-(2r-4+p)+1$ where each copy of $x_1$ (not including itself) forms a part of $G_k$ of size one. Clearly, we have $d(x) = 2r-4$ for all $x \in X = \{x_1,x_2,\dots,x_k\}$ and $e(X,X^c) = k(2r-4)$. Since $G_{2r-4+p}$ is $K_r$-free, so is $G_k$.

It remains to check that the admissible non-edges inside $X$, and those between $X$ and $X^c$ are $K_r$-saturated. Any admissible non-edge inside $X$ which is not inside the blow up class of $x_1$ is $K_r$-saturated by the same property of $G_{2r-4+p}$. Any admissible non-edge inside the blow up class of $x_1$ is $K_r$-saturated since $N(x_1)$ contains a $K_{r-2}$ by the construction of $G_{2r-4+p}$. Any admissible non-edge $x_iy_j$ where $j\not=1$ or ($j=1$ and $i \le 2r-4+p$), is $K_r$-saturated by the same property of $G_{2r-4+p}$. Any admissible non-edge $x_iy_j$ where $j=1$ and $2r-3+p \le i \le k$, is $K_r$-saturated since $N(x_1)\cap N(y_1)$ contains a $K_{r-2}$ by the construction of $G_{2r-4+p}$.
\qed

\section{Proof of Theorem~\ref{thm:alpha}\texorpdfstring{$(iii)$}{(iii)}}
\label{sec:alpha(iii)}

In this section, we study $\alpha(k,r)$ for $r=3,4,5$. The values of $\alpha(k,3)$ and $\alpha(k,4)$ are completely determined while the values of $\alpha(k,5)$ are unknown for $k=5,6,8$.

\subsection{The function \texorpdfstring{$\alpha(k,3)$}{alpha(k,3)}}

We shall prove that $\alpha(k,3) = 3(k-1)$ for $k \ge 3$. The upper bound follows from Theorem~\ref{thm:alpha}$(i)$. Let us prove the lower bound.

Let $G = V_1\cup V_2\cup \dots \cup V_k$ be a $K_3$-partite-saturated $k$-partite graph $G$ containing an independent set $X = \{x_1,x_2,\dots,x_k\}$ with $x_i \in V_i$ for all $i$. By Lemma~\ref{lem:deg}, the deletion of any part of $G$ does not destroy all vertices of $N(x_i)$ for all $i$, i.e. $x_i$ is joined to at least two parts of $G$. Suppose for contradiction that $e(X,X^c) < 3(k-1)$, i.e. $X$ contains at least four vertices of degree $2$, say $x_1,x_2,x_3,x_4$. Let $y_i \in V_i$ and $y_j \in V_j$ with $1 < i < j \le k$ be the neighbors of $x_1$, and so $y_i$ and $y_j$ are not neighbors otherwise $x_1y_iy_j$ forms a triangle. Since $\{2,3,4\} \setminus \{i,j\} \not= \emptyset$, we may assume that $i,j \not= 2$, i.e. $x_1,x_2,y_i,y_j$ are from different parts of $G$. Since any pair in $X$ forms a $K_3$-saturated non-edge in $G$, they have a common neighbor. So $x_1$ and $x_2$ have a common neighbor, say $y_i$.

First we suppose that $x_2y_j$ is a non-edge. Then $x_2$ and $y_j$ have a common neighbor $y_l \in V_l$. Since $y_i$ and $y_j$ are not neighbors, $l \not= i$. We obtain a contradiction by observing that $x_iy_jy_l$ forms a triangle. We observe that $x_iy_j$ are neighbors since $x_1$ and $x_i$ have a common neighbor and $N(x_1)=\{y_i,y_j\}$. Similarly, $x_iy_l$ are neighbors since $x_2$ and $x_i$ have a common neighbor and $N(x_2)=\{y_i,y_l\}$.

Now, suppose that $x_2y_j$ is an edge, and so $N(x_1)=N(x_2)=\{y_i,y_j\}$. Then $x_iy_j$ are neighbors since $x_1$ and $x_i$ have a common neighbor. Similarly, $x_jy_i$ are neighbors. We know that $x_i$ and $x_j$ have a common neighbor $y_l$ with $l \not= i,j$. Then either $l \not= 1$ or $l \not= 2$, say $l \not= 1$. Since the non-edge $x_1y_l$ is $K_3$-saturated, $y_l$ is joined to either $y_i$ or $y_j$. This implies a contradiction that either $x_jy_iy_l$ or $x_iy_jy_l$ forms a triangle.
\qed

\subsection{The function \texorpdfstring{$\alpha(k,4)$}{alpha(k,4)}}

As a consequence of Theorem~\ref{thm:alpha}$(ii)$, we obtain that $\alpha(k,4) = 4k$ for $k \ge 5$. For the remaining case $k=4$, we have the bounds $16 \le \alpha(4,4) \le 18$ from Theorem~\ref{thm:alpha}$(i)$. We shall show that $\alpha(4,4) = 18$.

Consider the family of graphs appearing in the definition of $\alpha(r,r)$. Let $G = V_1\cup V_2\cup \dots \cup V_r$ be an $K_r$-partite-saturated $r$-partite graph $G$ containing an independent set $X = \{x_1,x_2,\dots,x_r\}$ with $x_i \in V_i$ for all $i$. We shall establish some properties of $G$ which will be useful in this subsection, the next subsection and Section~\ref{sec:diagonal}.

We say that a vertex $y \in X^c$ is \emph{$i$-special} if $y$ is the only neighbor of $x_i$ in the part of $G$ containing $y$. The \emph{special degree} of a vertex $y \in X^c$ is the number of $i \in [r]$ such that $y$ is $i$-special. We say that a vertex $y \in X^c$ is \emph{special} if the special degree of $y$ is at least one. Let us make some easy observations regarding the special vertices.

\begin{lemma}\label{lem:special}
Let $G = V_1\cup V_2\cup \dots \cup V_r$ be an $K_r$-partite-saturated $r$-partite graph $G$ containing an independent set $X = \{x_1,x_2,\dots,x_r\}$ with $x_i \in V_i$ for all $i$. The following hold for $r \ge 4$.
    \begin{enumerate}[label={$(\roman*)$}, ref={\thelemma$(\roman*)$}]
        \item A special vertex $y_i \in V_i$ is joined to every vertex of $X$ except $x_i$. \label{lem:special:i}
		\item Each $V_i$ contains at most one special vertex.\label{lem:special:ii}
		\item If $y_i \in V_i$ is $i'$-special and $y_j \in V_j$ is $j'$-special with $i' \not= j$ and $j' \not= i$ then $y_iy_j$ is an edge.\label{lem:special:iii}
		\item The number of vertices of special degree at least $2$ is at most $r-2$.\label{lem:special:iv}
		\item If $y_i \in V_i$ is $i'$-special and $y_j \in V_j$ with $j \not= i,i'$ then $y_j$ is joined to either $y_i$ or $x_{i'}$.\label{lem:special:v}
        \item For a special vertex $y_i \in V_i$, there exist parts $V_j$ and $V_l$ where $i, j, l$ are distinct such that $N(x_{i}) \cap V_j$ and $N(x_{i}) \cap V_l$ both contain a non-neighbor of $y_i$. \label{lem:special:vi}
	\end{enumerate}
\end{lemma}

\begin{proof}
$(i)$ Let $y_i \in V_i$ be $i'$-special and let $j \not= i,i'$. Since the non-edge $x_{i'}x_j$ is $K_r$-saturated, the common neighborhood of $x_{i'}$ and $x_j$ contains a $K_{r-2}$ consisting of one vertex from each part of $G \setminus \left(V_{i'} \cup V_j\right)$. Then $y_i$ is in this $K_{r-2}$ since $y_i$ is the only neighbor of $x_{i'}$ in $V_i$, and so $y_i$ is joined to $x_j$.

$(ii)$ Suppose for contradiction that $V_i$ contains two special vertices $y_i$ and $z_i$ where $y_i$ is $i'$-special. Then, by $(i)$, $x_{i'}$ is joined to both $y_i$ and $z_i$ contradicting the fact that $y_i$ is the only neighbor of $x_{i'}$ in $V_i$.

$(iii)$ First, suppose that $i' \not= j'$. Since the non-edge $x_{i'}x_{j'}$ is $K_r$-saturated, the common neighborhood of $x_{i'}$ and $x_{j'}$ contains a $K_{r-2}$ consisting of one vertex from each part $G \setminus \left(V_{i'} \cup V_{j'}\right)$. Since $y_i$ is the only neighbor of $x_{i'}$ in $V_i$ and $y_j$ is the only neighbor of $x_{j'}$ in $V_j$, both $y_i$ and $y_j$ lie in this $K_{r-2}$. In particular, $y_iy_j$ is an edge.

Now, suppose that $i' = j'$. We can pick $l \not = i,j,i'$ because $r \ge 4$. Since the non-edge $x_{i'}x_l$ is $K_r$-saturated, the common neighborhood of $x_{i'}$ and $x_l$ contains a $K_{r-2}$ consisting of one vertex from each part of $G \setminus \left(V_{i'} \cup V_l\right)$. Since $y_i$ is the only neighbor of $x_{i'}$ in $V_i$ and $y_j$ is the only neighbor of $x_{i'}$ in $V_j$, both $y_i$ and $y_j$ lie in this $K_{r-2}$. In particular, $y_iy_j$ is an edge.

$(iv)$ Suppose for contradiction that there exist vertices $y_1,y_2,\dots,y_{r-1}$ of special degree at least $2$. By $(ii)$, they lie in different parts of $G$, say $y_i \in V_i$ for $1 \le i \le r-1$. We claim that they form a $K_{r-1}$ which would be a contradiction since, together with $x_r$, they form a $K_r$ by $(i)$. Now we show that any $y_iy_j$ is an edge. Since $y_i$ and $y_j$ have special degree at least $2$, there exist $i' \not= j$ and $j' \not= i$ such that $y_i$ is $i'$-special and $y_j$ is $j'$-special. Therefore, $y_iy_j$ is an edge by $(iii)$.

$(v)$ Suppose that $x_{i'}y_j$ is a non-edge. Then the common neighborhood of $x_{i'}$ and $y_j$ contains a $K_{r-2}$ consisting of one vertex from each part of $G \setminus \left(V_{i'} \cup V_j\right)$. Then $y_i$ is in this $K_{r-2}$ since $y_i$ is the only neighbor of $x_{i'}$ in $V_i$, and so $y_i$ is joined to $y_j$.

$(vi)$ Suppose for contradiction that there exists $j \in [r] \setminus \{i\}$ such that $y_i \in V_i$ is joined to every vertex in $N(x_{i}) \cap V_l$ for all $l \not= i,j$. Since the non-edge $x_{i}x_j$ is $K_r$-saturated, the common neighborhood of $x_{i}$ and $x_j$ contains a $K_{r-2}$ consisting of one vertex from each part of $(G \setminus X) \setminus \left(V_{i} \cup V_j\right)$. We obtain a contradiction by observing that this $K_{r-2}$, together with $x_j$ and $y_i$, form a $K_r$. Indeed, by assumption, this $K_{r-2}$ is also in the neighborhood of $y_i$ and $x_jy_i$ is an edge by $(i)$.
\end{proof}

Now we are ready to show that $\alpha(4,4) \ge 18$. Suppose for contradiction that $\alpha(4,4) \le 17$, i.e. there exists a $K_4$-partite-saturated $4$-partite graph $G = V_1\cup V_2\cup V_3\cup V_4$ containing an independent set $X = \{x_1,x_2,x_3,x_4\}$ with $x_i \in V_i$ for all $i$ such that $\sum_{i=1}^4 d(x_i) \le 17$. By Lemma~\ref{lem:deg}, $d(x_i) \ge \beta_1(3,3) = 4$ and each $x_i$ has some neighbor in $V_j$ for $j \not= i$. Therefore, there are at least three vertices of degree $4$ and possibly one of degree $5$. Since a vertex of degree $4$ in $X$ creates at least two special vertices and a vertex of degree $5$ in $X$ creates at least one special vertex, the sum of the special degrees of the vertices in $X^c$ is at least $2+2+2+1=7$. By Lemma~\ref{lem:special:iv}, there is a vertex of special degree $3$, say $y_1 \in V_1$.

For $i=2,3,4$, since $y_1$ is $i$-special, $x_i$ has at least three neighbors in $N(y_1) \cup \{y_1\}$, each in a different part of $G$, by Lemma~\ref{lem:deg}. On the other hand, $y_1$ has at least two non-neighbors, say $y_2 \in V_2$ and $y_3 \in V_3$, by Lemma~\ref{lem:special:vi}. By Lemma~\ref{lem:special:v}, $x_iy_2$ is an edge for $i \not= 2$ and $x_iy_3$ is an edge for $i \not= 3$. So $x_4$ has five neighbors, i.e. $y_2$, $y_3$ and three vertices in $N(y_1) \cup \{y_1\}$, and $d(x_1)=d(x_2)=d(x_3)=4$. Since $x_2$ has four neighbors including $y_3$ and it has some neighbor in $(N(y_1) \cup \{y_1\}) \cap V_j$ for each $j=1,3,4$, it has exactly one neighbor in $V_4$, say $y_4$. Similarly, $x_3$ has exactly one neighbor in $V_4$ which has to be the same vertex $y_4$ by Lemma~\ref{lem:special:ii}.

We obtain a contradiction by observing that $x_1y_2y_3y_4$ forms a $K_4$. First, note that $x_1y_4$ is an edge by Lemma~\ref{lem:special:i}. Now $y_4$ is not $1$-special otherwise $y_4$ would have special degree $3$ and by repeating the argument above with $y_1$ replaced by $y_4$, we could deduce that $x_1$, $x_2$, or $x_3$ had degree $5$. Therefore, the neighbors of $x_1$ are $y_2$, $y_3$, $y_4$ and a vertex in $V_4$. 
Since $y_2,y_3$ are both $1$-special and $y_4$ is $2,3$-special, $y_2y_3y_4$ forms a triangle by Lemma~\ref{lem:special:iii}.
\qed

\subsection{The function \texorpdfstring{$\alpha(k,5)$}{alpha(k,5)}}

As a consequence of Theorem~\ref{thm:alpha}$(i)$ and $(ii)$, we obtain that
	\begin{align*}
		\alpha(k,5) &= 6k \quad\text{ for } k=7 \text{ or } k \ge 9,\\
		&30 \le \alpha(5,5) \le 36,\\
		&36 \le \alpha(6,5) \le 40,\\
		&48 \le \alpha(8,5) \le 49.
	\end{align*}
We shall improve the lower bound for $\alpha(5,5)$ to $33$.


    Suppose for contradiction that $\alpha(5,5) \le 32$, i.e. there exists a $K_5$-partite-saturated $5$-partite graph $G = V_1\cup V_2\cup V_3\cup V_4\cup V_5$ containing an independent set $X = \{x_1,x_2,x_3,x_4,x_5\}$ with $x_i \in V_i$ for all $i$ such that $\sum_{i=1}^5 d(x_i) \le 32$. 
    Write $Y_{i}$ for $V_{i} \setminus \left\{ x_{i} \right\}$.
    By Lemma~\ref{lem:deg}, $d(x_i) \ge \beta_1(4,4) = 6$ and each $x_i$ has some neighbor in $V_j$ for $j \not= i$. 
    Therefore, there are either four vertices in $X$ of degree $6$ or there are three vertices of degree $6$ and two of degree $7$.
    Since a vertex of degree $6$ in $X$ creates at least two special vertices and a vertex of degree $7$ in $X$ creates at least one special vertex, the sum of the special degrees of the vertices in $X^c$ is at least $8$, and hence, there exists a vertex of special degree at least two.
    Let $i$ be such that there is a special vertex $y \in Y_{i}$ with special degree $d_{s}(y)$ at least two where $(d(x_{i}), d_{s}(y))$ is maximum in lexicographical order\footnote{We say that $(a, b) \preccurlyeq (c, d)$ if $a < c$ or $a = c$ and $b \le d$, where $\preccurlyeq$ denotes the lexicographical order relation.}.
    Without loss of generality we can assume that $i = 1$.
    Let $N = N(y) \setminus X$.
    By Lemma~\ref{lem:special:vi}, $x_{1}$ has two neighbours, say $y_{2}, y_{3}$, belonging to two distinct parts of $G$, different from $V_{1}$, which are non-neighbours of $y$.
    Without loss of generality, we can assume that $y_{2} \in Y_{2}$ and $y_{3} \in Y_{3}$.

    For a pair of non-adjacent vertices $u, v \in G$ and $S \subset G$, we say that $S$ is $uv$-\emph{saturating} if adding the edge of $uv$ to $G$ creates a copy $K$ of $K_{5}$ such that $S \subseteq K$.
    If $S = \left\{ z \right\}$ then we simply say that $z$ is $uv$-saturating.
    Notice that if $S$ is $uv$-saturating then $S$ induces a clique.
    
    In the rest of the proof, we shall repeatedly use the following lemma.
    \begin{lemma} \label{lemma33}
        Given  $i \in \left\{ 2, 3, 4, 5 \right\}$ the following hold.
        \begin{enumerate}[label={$(\roman*)$}, ref={\thelemma$(\roman*)$}]
            \item \label{lemma33:i}     If $j \in \left\{ 2, 3, 4, 5 \right\} \setminus \left\{ i \right\}$ then $x_{i}$ has a neighbour in $V_{j} \cap N$.
                                        In particular, $d_{N}(x_{i}) \ge 3$.
            \item \label{lemma33:ii}    If $y$ is $i$-special then $x_{i}$ is adjacent to $y_{j}$ for every $j \in \left\{ 2, 3 \right\} \setminus \left\{ i \right\}$.
            \item \label{lemma33:iv}    If $y$ is $x_{i}x_{j}$-saturating, for every $j \in \left\{ 2, 3, 4, 5 \right\} \setminus \left\{ i \right\}$, then $d_{N}(x_{i}) \ge 4$.
            \item \label{lemma33:iii}   If $y$ is $i$-special or $d_{s}(y) \ge 3$ then $d_{N}(x_{i}) \ge 4$.
            \item \label{lemma33:v}     If $y$ is $2,3$-special and $i \in \left\{ 4, 5 \right\}$, then $d(x_{i}) \ge 7$.
            \item \label{lemma33:vi}    If $i \in \left\{ 2, 3 \right\}$ and there are $p$ vertices in $X \setminus \left\{ x_{1} \right\}$ all of which have neighbours in $Y_{i} \setminus N$ then there is no vertex in $V_{i}$ with special degree bigger than $\max\left\{ 1, 3-p \right\}$.
            \item \label{lemma33:vii}   $Y_{3} \cup Y_{4} \subset N$.
        \end{enumerate}
    \end{lemma}
    \begin{proof}
        $(i)$ Observe that we can choose $k \in \left\{ 2, 3, 4, 5 \right\} \setminus \left\{ i, j \right\}$ such that $y$ is either $i$-special or $k$-special.
        Since there must be a triangle in the common neighbourhood of $x_{i}$ and $x_{j}$ which uses $y$, we have that the remaining two vertices belong to $N$.
        Hence $x_{i}$ has a neighbour in $N \cap x_{j}$.

        $(ii)$ This follows directly from Lemma~\ref{lem:special:v}.

        $(iii)$
        We shall show that $d_{N}(x_{i}) \ge \beta_{1}(3, 3) = 4$.
        Take any $j \in \left\{ 2, 3, 4, 5 \right\} \setminus \left\{ i \right\}$.
        Since $y$ is $x_{i}x_{j}$-saturating then there is an edge in the common neighbourhood of $x_{i}$ and $x_{j}$ in $N \setminus \left( V_{i} \cup V_{j} \right)$.
        Observe that the common neighbourhood of $x_{i}$ and $y$ cannot contain a $K_{3}$, hence $d_{N}(x_{i}) \ge \beta_{1}(3, 3) = 4$.

        $(iv)$ Take any $j \in \left\{ 2, 3, 4, 5 \right\} \setminus \left\{ i \right\}$. 
        Since $y$ is either $i$- or $j$-special, it follows that $y$ is $x_{i}x_{j}$-saturating.
        Hence, by $(ii)$, $d_{N}(x_{i}) \ge 4$.

        $(v)$ 
        Without loss of generality we can assume that $i = 4$.
        If $y$ is also $4$-special then it follows from $(ii)$ and $(iv)$ that $d_{N}(x_{4}) \ge 4$ and $x_{4}$ is adjacent to $y, y_{2}, y_{3}$, therefore $d(x_{4}) \ge 7$.
        Hence we can assume that $y$ is not $4$-special.
        Suppose for contradiction that $d(x_{4}) = 6$.
        From $(i)$, we have that $d_{N}(x_{4}) \ge 3$ and since $y$ is not $4$-special we have that $d_{Y_{1}}(x_{4}) \ge 2$.
        Moreover, $x_{4}$ has to have at least one neighbour not in $Y_{1} \cup N$ as otherwise there would be a copy of $K_{5}$ in $G$, as seen by considering the non-edge $x_{1}x_{4}$.
        Therefore, $d(x_{4}) = d_{Y_{1}}(x_{4}) + d_{N}(x_{4}) + |N(x_{4}) \setminus \left( Y_{1} \cup N \right)| \ge 3 + 2 + 1 = 6 = d(x_{4})$.
        Hence, $d_{Y_{1}}(x_{4}) = 3$, $d_{N}(x_{4}) = 4$ and $|N(x_{4}) \setminus \left( Y_{1} \cup N  \right)| = 1$. 
        We shall obtain a contradiction by finding a copy of $K_{5}$ in the graph $G$.
        
        Suppose $\left\{ z_{1}, z_{2}, z_{3} \right\}$ is $x_{4}x_{5}$ saturating, with $z_{i} \in V_{i}$.
        We claim that $y \neq z_{1}$ and $\left\{ z_{2}, z_{3} \right\} \not\subseteq N$.
        Suppose for contradiction that it is not the case.
        If $y$ is $x_{4}x_{5}$-saturating then from $(iii)$ we have that $d_{N}(x_{4}) \ge 4$ hence we obtain a contradiction.
        We can therefore assume that $y$ is not $x_{4}x_{5}$-saturating and hence $z_{1} \neq y$.
        Whence $z_{2}, z_{3} \in N$.
        Recall that $\left\{ z_{1}, z_{2}, z_{3} \right\}$ form a triangle and therefore there is an edge between $z_{2}, z_{3}$.
        By assumption $z_{2}$ and $z_{3}$ are neighbours of $y$, hence $y, z_{2}, z_{3}$ form a triangle, and therefore $y$ is $x_{4}x_{5}$-saturating since $y, z_{2}, z_{3}$ belong to the common neighbourhood of $x_{4}$ and $x_{5}$, which contradicts the assumption that $y$ is not $x_{4}x_{5}$-saturating.

        Without loss of generality we can assume that $z_{2} \not\in N$.
        Using $(i)$, we can therefore suppose that $N(x_{4}) \cap Y_{1} = \left\{ y, z_{1} \right\}$, $N(x_{4}) \cap Y_{2} = \left\{w, z_{2}\right\}$, $N(x_{4}) \cap Y_{3} = \left\{ z_{3} \right\}$ and $N(x_{4}) \cap Y_{5} = \left\{ z_{5} \right\}$, for some $w, z_{3}, z_{5} \in N$.
        We shall obtain a contradiction by observing that $z_{1}, z_{2}, z_{3}, x_{4}, z_{5}$ form a copy of $K_{5}$.
        First we claim that $\{z_{2}, z_{3}, z_{5}\}$ is $x_{1}x_{4}$-saturating.
        Indeed, there must be a triangle in the common neighbourhood of $x_{1}$ and $x_{4}$, with one vertex in each $V_{3}, V_{4}, V_{5}$.
        There are only two candidates for the triangle: $z_{2}, z_{3}, z_{5}$ or $w, z_{3}, z_{5}$.
        It cannot be $w, z_{3}, z_{5}$ since they are all neighbours of $y$, hence $y, w, z_{3}, x_{4}, z_{5}$ would form a copy of $K_{5}$.
        Hence we must have that the set $\{z_{2}, z_{3}, z_{5}\}$ is $x_{1}x_{4}$-saturating.
        Now, since $x_{4}$ is not adjacent to $y_{3}$, and $y_{3}$ is not adjacent to $y$ we must have an edge between $z_{1}$ and $z_{5}$.
        Indeed, there must be a triangle in the common neighbourhood of $x_{4}$ and $y_{3}$ with a vertex in each $V_{1}, V_{2}, V_{3}$.
        Since $x_{4}$ has only one neighbour in $V_{5}$, i.e. $z_{5}$, and $x_{4}$ and $y_{3}$ have only one common neighbour in $V_{1}$, i.e. $z_{1}$, we must have an edge between $z_{1}$ and $z_{5}$.

        Therefore we have that $z_{1}, z_{2}, z_{3}$ form a triangle, $z_{2}, z_{3}, z_{5}$ form a triangle, and $z_{1}, z_{5}$ are adjacent.
        It easy to see now that $z_{1}, z_{2}, z_{3}, x_{4}, z_{5}$ form a copy of $K_{5}$.

        $(vi)$ Let $v$ be a special vertex in $V_{2} \cup V_{3}$, say in $V_{2}$.
               First observe that if $v$ is $1$-special then $x_{3}, x_{4}, x_{5}$ are all adjacent to $y_{2} \in Y_{2} \setminus N$.
               On the other hand, it follows from $(i)$ that $x_{3}, x_{4}, x_{5}$ all have neighbours in $N \cap Y_{2}$ hence they all have degree at least $2$ in $Y_{2}$.
               It follows that $v$ has special degree $1$.
               If we assume that $v$ is not $1$-special then $v$ has special degree at most $3-p$, since $p$ of the vertices $x_{3}, x_{4}, x_{5}$ have degree $2$ in $Y_{2}$.

        $(vii)$ Assume for contradiction that there is $v$, say in $Y_{4} \setminus N$.
                Observe that if $y$ is $i$-special then it follows from $(ii)$ and $(iv)$ that $d(x_{i}) \ge 7$, hence if $d_{s}(y) \ge 3$ we obtain contradiction by finding three vertices in $X$ of degree at least $7$.
                Therefore we can assume that $d_{s}(y) = 2$.

                If $y$ is $5, i$-special, then from $(ii)$ and $(iv)$ we have that $d(x_{5}) \ge 8$ and $d(x_{i}) \ge 7$ hence again we obtain a contradiction.
                Therefore we can assume that $y$ is not $5$-special.
                If $y$ is $2,3$-special then $d(x_{2}), d(x_{3}) \ge 7$ and from $(iv)$ we have that $d(x_{4}), d(x_{5}) \ge 7$.
                Hence we can assume that $y$ is $2, 4$-special or $3, 4$-special.
                Suppose that the former is the case.
                Then $d(x_{2}), d(x_{4}) \ge 7$.
                It follows that $d(x_{1}) = 6$.
                Therefore by maximality $(x_{1}, y)$ and from $(v)$ we have that every vertex in $Y_{2} \cup Y_{3} \cup Y_{4}$ has special degree at most $1$ and no vertex in $Y_{5}$ has special degree bigger than $2$.
                Which gives a contradiction since the sum of special degree is then at most $7$.
    \end{proof}

    We are now ready to finish showing that $\alpha(5, 5) \ge 33$.
    We consider several cases depending on the special degree of $y$.
    \mycase{$d_{s}(y) = 4$}
        Consider the $4$-partite graph $H = G[N(y)]$ with an independent set $X' = \{x_2,x_3,x_4,x_5\}$. 
        Clearly, $H$ is $K_4$-free since $G$ is $K_5$-free. 
        We modify $H$ by keeping adding admissible edges inside $H \setminus X'$ until every admissible non-edge inside $H \setminus X'$ is $K_4$-saturated. 
        We claim that $H$ is $K_4$-partite-saturated, which would imply that $e(X',H \setminus X') \ge \alpha(4,4) = 18$ by the previous subsection. 
        It remains to show that the admissible non-edges with at least one endpoint in $X'$ are $K_4$-saturated.

        Consider the non-edge $x_iy_j$ with $y_j \in V_j \cap H$ (possibly $y_j=x_j$) and distinct $2 \le i,j \le 5$. 
        Since the non-edge $x_iy_j$ is $K_5$-saturated in $G$, the common neighborhood in $G$ of $x_i$ and $y_j$ contains a $K_3$ consisting of one vertex from each part of $G \setminus \left(V_i \cup V_j\right)$. 
        Since $y$ is $i$-special, this $K_3$ must contain $y$, and so the common neighborhood in $H$ of $x_i$ and $y_j$ contains a $K_2$, i.e. $x_iy_j$ is $K_4$-saturated in $H$ as required. 

        Recall that $y$ has two non-neighbors,  $y_2 \in V_2$ and $y_3 \in V_3$.
        By Lemma~\ref{lem:special:v}, $x_iy_2$ is an edge for $i \not= 2$ and $x_iy_3$ is an edge for $i \not= 3$. 
        We shall partition the edges between $X$ and $X^c$ as follows:
            \begin{align*}
                e(X,X^c)
                &\ge e(X',H \setminus X') + d(x_1) + e(X, y) + e(X', y_2) + e(X', y_3) \\
                &\ge 18 + 6 + 4 + 3 + 3 = 34,
            \end{align*}
        contradicting the assumption.
    \mycase{$d_{s}(y) = 3$}
        If $y$ is $4,5$-special then from Lemma~\ref{lemma33:ii}~and~\ref{lemma33:iv} we have that $d(x_{4}), d(x_{5}) \ge 7$.
        Otherwise $y$ is $2,3$-special and hence it follows from Lemma~\ref{lemma33:v} that $d(x_{4}), d(x_{5}) \ge 7$.
        We shall obtain a contradiction by showing that $d(x_{1}) \ge 7$, hence showing that there are three vertices in $X$ with degrees at least $7$, which is against an assumption made in the beginning of the subsection.
        It follows from Lemma~\ref{lemma33:vi} with $p \ge 2$, that the sum of special degrees in $Y_{2} \cup Y_{3}$ is at most $2$.
        Since the sum of special degrees is at least $8$, it follows that there is a special vertex in $Y_{4} \cup Y_{5}$ with special degree at least $2$.
        Therefore from the maximality of $d(x_{1})$ we have that $d(x_{1}) \ge 7$.
    \mycase{$d_{s}(y) = 2$}
        We split this case into three subcases. 
        \mysubcase{$y$ is $2, 3$-special}
            It follows from Lemma~\ref{lemma33:v} that $d(x_{4}), d(x_{5}) \ge 7$.
            We shall obtain a contradiction by showing that $d(x_{1}) \ge 7$, hence showing that there are three vertices in $X$ with degrees at least $7$, which is against an assumption made in the beginning of the subsection.
            It follows from Lemma~\ref{lemma33:vi} that the sum of special degrees in $Y_{2} \cup Y_{3}$ is at most $2$.
            Since the sum of special degrees is at least $8$, it follows that there is a special vertex in $Y_{4} \cup Y_{5}$ with special degree at least $2$.
            Therefore from the maximality of $d(x_{1})$ we have that $d(x_{1}) \ge 7$.
        \mysubcase{$y$ is $4, 5$-special}
            It follows from Lemma~\ref{lemma33:ii}~and~\ref{lemma33:iii} that $d(x_{4}),d(x_{5}) \ge 7$.
            We shall obtain a contradiction by showing that $d(x_{1}) \ge 7$, hence showing that there are three vertices in $X$ with degrees at least $7$, which is against an assumption made in the beginning of the subsection.
            It follows from Lemma~\ref{lemma33:vi} that the sum of special degrees in $Y_{2} \cup Y_{3}$ is at most $3$.
            Since the sum of special degrees is at least $8$, it follows that there is a special vertex in $Y_{4} \cup Y_{5}$ with special degree at least $2$.
            Therefore from the maximality of $d(x_{1})$ we have that $d(x_{1}) \ge 7$.
        \mysubcase{$y$ is neither $2, 3$-special nor $4, 5$-special}
            Without loss of generality we can assume that $y$ is $2, 4$-special.
            It follows from Lemma~\ref{lemma33:ii}~and~\ref{lemma33:iii} that $d(x_{4}) \ge 7$ and from Lemma~\ref{lemma33:vi} with $p \ge 2$ that there is no special vertex in $Y_{2} \cup Y_{3}$ with special degree bigger than $1$.
            Hence there is either a vertex in $Y_{4}$ with special degree at least $2$ or a vertex in $Y_{5}$ with special degree at least $3$.
            Therefore we can assume that $d(x_{1}) = 7$ as otherwise we obtain a contradiction to the maximality of $\left( d(x_{1}), d_{s}(y) \right)$.
            
            We shall obtain a contradiction by showing that at least one of $x_{2}$, $x_{3}$ or $x_{5}$ has degree at least $7$, thus finding three vertices with degree at least $7$.
            Suppose $d(x_{2}) = d(x_{3}) = d(x_{5}) = 6$.
            Observe that if there is a vertex in $X_{4}$ of special degree bigger than $2$ then we obtain a contradiction to the maximality of $\left( d(x_{1}), d_{s}(y) \right)$.
            Therefore there are two vertices in $X$ with at least two neighbours in $X_{4}$.
            Suppose that $i \in \{3, 5 \}$ and $x_{i}$ has at least two neighbours in $X_{4}$.
            Then it follows from Lemma~\ref{lemma33:i} that $d_{N}(x_{i}) \ge 4$, and hence $x_{i}$ has degree at least $7$ as $x_{i}$ has at least three neighbours outside $N$.
            We can therefore assume that $x_{3}$ and $x_{5}$ have only one neighbour in $X_{4}$.
            For the same reason we can assume that $x_{3}$ has only one neighbour in $Y_{5}$.
            If $x_{2}$ has two neighbours in $Y_{5}$ then $d_{N}(x_{2}) \ge 5$ and therefore $d(x_{2}) \ge 7$.
            Hence we can assume that there is $z_{5} \in Y_{5}$ which is $2,3$-special.

            Suppose $\left\{ z_{1}, z_{2}, z_{4} \right\}$ is $x_{3}x_{5}$-saturating, with $z_{i} \in V_{i}$.
            We claim that $y \neq z_{1}$ and $z_{2} \not\in N$.
            Suppose for contradiction that it is not the case.
            If $y$ is $x_{3}x_{5}$-saturating then from $(iii)$ we have that $d_{N}(x_{3}) \ge 4$ hence we obtain a contradiction.
            We can therefore assume that $y$ is not $x_{3}x_{5}$-saturating and hence $z_{1} \neq y$.
            Whence $z_{2} \in N$.
            Observe that by Lemma~\ref{lemma33:vii} we have $z_{4} \in N$.
            Recall that $\left\{ z_{1}, z_{2}, z_{4} \right\}$ form a triangle and therefore there is an edge between $z_{2}, z_{4}$.
            By assumption $z_{2}$ and $z_{4}$ are neighbours of $y$, hence $y, z_{2}, z_{4}$ form a triangle, and therefore $y$ is $x_{3}x_{5}$-saturating since $y, z_{2}, z_{4}$ belong to the common neighbourhood of $x_{3}$ and $x_{5}$, which contradicts the assumption that $y$ is not $x_{3}x_{5}$-saturating.

            We shall obtain a contradiction by showing that $z_{1}, z_{2}, x_{3}, z_{4}, z_{5}$ form a copy of $K_{5}$.
            Indeed, by assumption $\left\{ z_{1}, z_{2}, z_{4} \right\}$ is $x_{3}x_{5}$-saturating and similar analysis to the one made in the proof of Lemma~\ref{lemma33:v} shows that $\left\{z_{2}, z_{4}, z_{5}\right\}$ is $x_{1}x_{3}$-saturating.
            Since $y$ is $2$-special it follows that $x_{2}$ is not adjacent to $z_{1}$, and moreover $z_{5}$, as the only neighbour of $x_{2}$ in $Y_{5}$, is $x_{2}z_{1}$-saturating, and therefore there is an edge between $x_{2}$ and $z_{5}$.
            Hence we have that $z_{2}, z_{4}, z_{5}$ form a triangle, $z_{1}, z_{2}, z_{4}$ form a triangle, and $z_{1}, z_{5}$ are adjacent.
            It easy to see now that $z_{1}, z_{2}, x_{3}, z_{4}, z_{5}$ form a copy of $K_{5}$.

\qed

\section{The diagonal case \texorpdfstring{$\alpha(r,r)$}{alpha(r,r)}}
\label{sec:diagonal}

\subsection{Proof of Theorem~\ref{thm:alpha}\texorpdfstring{$(iv)$}{(iv)}}

We have seen that the lower bound $\alpha(k,r) \ge k(2r-4)$ in Theorem~\ref{thm:alpha}$(i)$ is attained for some $k$. In this subsection, we show that this is not the truth for the diagonal case $k=r\ge 4$, i.e. $\alpha(r,r) \ge r(2r-4) + 1$. We shall again use the concept of special vertices introduced in Section~\ref{sec:alpha(iii)}.

Suppose for contradiction that for some $r \ge 4$, $\alpha(r,r) = r(2r-4)$, i.e. there exists a $K_r$-partite-saturated $r$-partite graph $G = V_1\cup V_2\cup \dots \cup V_r$ containing an independent set $X = \{x_1, x_2,\dots, x_r\}$ with $x_i \in V_i$ for all $i$ such that $\sum_{i=1}^r d(x_i) = r(2r-4)$. Lemma~\ref{lem:deg} tells us that we must have $d(x_i) = 2r-4$ for all $i$ and each $x_i$ has some neighbor in $V_j$ for $j \not= i$. Therefore, each $x_i$ creates at least two special vertices, and so the sum of the special degrees of the vertices in $X^c$ is at least $2r$. By Lemma~\ref{lem:special:iv}, there is a vertex of special degree at least $3$, say $y_1 \in V_1$.

We observe that $y_1$ has at least two non-neighbors, say $y_2 \in V_2$ and $y_3 \in V_3$ by Lemma~\ref{lem:special:vi}. Since $y_1$ has special degree at least $3$, we can pick $i \ge 4$ such that $y_1$ is $i$-special. By Lemma~\ref{lem:special:v}, $y_2$ and $y_3$ are neighbors of $x_i$. Therefore,
	\[
		|N(x_i) \cap N(y_1)| = d(x_i) - |N(x_i) \setminus N(y_1)|
		\le (2r-4) - 3 = 2r-7.
	\]
On the other hand, we shall obtain a contradiction by showing that the graph $H = G[N(x_i) \cap N(y_1)]$ contains at least $\beta_1(r-2,r-2) = 2(r-3)$ vertices. It is sufficient to prove that $H$ is an $(r-2)$-partite $K_{r-2}$-free graph such that the subgraph induced by any $k-3$ parts contains a $K_{r-3}$. Clearly, $H$ is $K_{r-2}$-free since $G$ is $K_r$-free. The parts of $H$ are $N(x_i)\cap N(y_1)\cap V_j$ for $j \not\in [r] \setminus \{1,i\}$. It remains to verify that the deletion of the part $N(x_i)\cap N(y_1)\cap V_j$ does not destroy all the $K_{r-3}$. Since the non-edge $x_ix_j$ is $K_r$-saturated in $G$, the common neighborhood in $G$ of $x_i$ and $x_j$ contains a $K_{r-2}$ consisting of one vertex from each part of $G \setminus \left(V_i \cup V_j\right)$. Since $y_1$ is $i$-special, this $K_{r-2}$ must contain $y_1$, and so the common neighborhood $N(x_i)\cap N(y_1)\cap N(x_j) \subset H$ contains a $K_{r-3}$ not using the vertices of $V_j$ as required.
\qed

\subsection{Remark on \texorpdfstring{$\beta_2(r,r-1)$}{beta2(r,r-1)}}

Recall from Proposition~\ref{prop:alpha} that $\alpha(r,r) \le (r-1)\beta_2(r,r-1)$. Thus, a better estimate on $\beta_2$ would translate to a better understanding of the saturation numbers. While we could not find the exact value of $\beta_2(r,r-1)$, we suspect that
$\beta_2(r,r-1)=3r-6$ as mentioned in Conjecture~\ref{conj:beta2}. In this subsection, we make an observation about $\beta_2(r,r-1)$ which can be viewed as a first step towards determining its exact value. For simplicity of notation, let us write $\beta_2(r)=\beta_2(r,r-1)$.

\begin{proposition}
Either
\begin{itemize}
\item $\beta_2(r)=3r-6$ for all $r \ge 3$, or
\item $\beta_2(r)\leq (c+o(1))r$ for some constant $c<3$, as $r\rightarrow\infty$.
\end{itemize}
\end{proposition}

\begin{proof} 
The result is an immediate consequence of the following lemma.

\begin{lemma}
$\beta_2(r_1+r_2)\leq \beta_2(r_1)+\beta_2(r_2)+6$ for $r_1,r_2 \ge 3$.
\end{lemma}
\begin{proof}

For $i \in \{1,2\}$, let $G_i=V_{i,1}\cup V_{i,2}\cup \dots\cup V_{i,r_i}$ be a $K_{r_i-1}$-free $r_i$-partite graph on $\beta_2(r_i)$ vertices such that the subgraph induced by any $r_i-2$ parts contains a $K_{r_i-2}$. We shall construct a $K_{r_1+r_2-1}$-free $(r_1+r_2)$-partite graph $G$ from $G_1$ and $G_2$ with $|G|=|G_1|+|G_2|+6$ by starting with the disjoint union of $G_1$ and $G_2$ and then adding six new vertices $U=\{x_1,x_2,y_1,y_2,z_1,z_2\}$ as follows: add $x_i,y_i$ to $V_{i,1}$ and add $z_i$ to $V_{i,2}$ for $i \in \{1,2\}$. Now, join all admissible pairs between $U$ and $V(G)\setminus U$, and add the edges $x_1z_1,x_2z_2,y_1y_2,z_1z_2,y_1z_2,z_1y_2$ inside $U$.

First, we show that $G$ is $K_{r_1+r_2-1}$-free. Suppose otherwise. Since $G_i$ is $K_{r_i-1}$-free for $i \in \{1,2\}$, this $K_{r_1+r_2-1}$ must contain at least three vertices forming a triangle in $U$, contradicting the fact that $G[U]$ is triangle-free. It remains to show that the deletion of any two parts does not destroy all the $K_{r_1+r_2-2}$. Suppose first that both deleted parts are from $G_1$. Since $G_1$ contains a $K_{r_1-2}$ not using these two parts and $G_2$ contains a $K_{r_2-2}$ not using $V_{2,1}$ and $V_{2,2}$, we obtain a $K_{r_1+r_2-2}$ not using the deleted parts, formed by these two cliques and $x_2,z_2$. Now suppose that one of the deleted parts is from $G_1$ and the other is from $G_2$. For $i \in \{1,2\}$, let $V_i$ be a part in $\{V_{i,1},V_{i,2}\}$ which was not deleted. By construction, $G[U]$ contains an edge between $V_{1,j}$ and $V_{1,l}$ for all $j,l \in \{1,2\}$ and so there exists an edge in $G[U]$ between $V_1$ and $V_2$, say $e$. Since $G_1$ contains a $K_{r_1-2}$ not using the deleted part in $G_1$ and $V_1$, and $G_2$ contains a $K_{r_2-2}$ not using the deleted part in $G_2$ and $V_2$, we obtain a $K_{r_1+r_2-2}$ not using the deleted parts, formed by these two cliques and the endpoints of $e$.
\end{proof}

Suppose that $\beta_2(s)<3s-6$ for some $s \ge 3$. We shall show that $\beta_2(r) \le (c+o(1))r$ with $c=\frac{\beta_2(s)+6}{s}<3$. Applying the lemma and induction on $m$, we deduce that $\beta_2(ms) \le cms-6$ for all positive integer $m$. Hence, writing $r=ms+t$ with $3 \le t \le s+2$ and applying the lemma again,
\[
\beta_2(r) \le \beta_2(ms)+\beta_2(t)+6 \le cms+d \le \left(c+\frac{d}{r}\right)r = (c+o(1))r
\]
where $d=\max\{\beta_2(t): 3 \le t \le s+2\}$.
\end{proof}

\section{Proof of Theorem~\ref{thm:constant}}
\label{sec:thm:constant}

Theorems~\ref{thm:sat} and Theorem~\ref{thm:alpha}$(ii)$ imply that \[sat(n,k,r) = k(2r-4)n + o(n) \text{ if }
		\begin{cases}
				k = 2r-3,\text{ or}\\
				k \ge 2r-2 \text{ and }r \equiv 0 \mod 2,\text{ or}\\
				k \ge 2r-1 \text{ and }r \equiv 2 \mod 3.
		\end{cases}\]
In this section, we shall show that the $o(n)$ term can be replaced with $O(1)$. The upper bound follows from Proposition~\ref{prop:satupper} and Theorem~\ref{thm:alpha}$(ii)$. We prove that the lower bound holds for any $k \ge r \ge 3$ using the fact that $\beta_1(k-1,r-1) = 2r-4$.

\begin{proposition}
For $k \ge r \ge 3$, there is an integer $C_{k,r}$ such that $sat(n,k,r) \ge k(2r-4)n + C_{k,r}$, for every integer $n \ge 0$.
\end{proposition}

\begin{proof}
Suppose, as we may, that $n$ is sufficiently large. Let $G = V_1 \cup V_2 \cup \dots \cup V_k$ be a $K_r$-partite-saturated $k$-partite graph with $|V_i|=n$ for all $i$. We shall find a subset $U$ of $V(G)$ of constant size such that every vertex in $U^c$ has at least $2r-4$ neighbors in $U$. Then we would be done since $e(G) \ge e(U,U^c) \ge (2r-4)(kn-|U|)$. Let $v_1$ be a vertex of smallest degree in $V_1$. Having defined $v_1,v_2,\dots,v_{i-1}$, let $v_i\in V_i$ be a vertex of smallest degree in $V_i\setminus \left(N(v_1)\cup N(v_2)\cup \dots \cup N(v_{i-1})\right)$. We shall take $U$ to be $N(v_1)\cup N(v_2)\cup \dots \cup N(v_k)$. Now we may assume that $d(v_i) < 2k(2r-4)$ for all $1\le i\le k$. Indeed, if $v_i$ is the first vertex in the sequence such that $d(v_i) \ge 2k(2r-4)$ then we are done since
	\[
		e(G) \ge e(V_i,V_i^c) \ge d(v_i)\left(n-\sum_{j<i} d(v_j)\right) \ge 2k(2r-4)\Big(n-2k(2r-4)(i-1)\Big) \ge k(2r-4)n
	\]
for sufficiently large $n$. Therefore, $U$ has size bounded by a function of $k$ and $r$. It remains to show that every vertex $v \in U^c$ has at least $2r-4$ neighbors in $U$. We shall prove that $H = G[N(v)\cap U]$ contains at least $\beta_1(k-1,r-1) = 2r-4$ vertices by showing that $H$ is a $K_{r-1}$-free $(k-1)$-partite graph such that the subgraph induced by any $k-2$ parts contains a $K_{r-2}$. Clearly, $H$ is $K_{r-1}$-free since $G$ is $K_r$-free. Without loss of generality, $v \in V_1$. The parts of $H$ are $N(v)\cap U\cap V_i$ for $2 \le i \le k$. The deletion of the part $N(v)\cap U\cap V_i$ does not destroy all the $K_{r-2}$ since the non-edge $vv_i$ is $K_{r-1}$-saturated in $G$, i.e. $N(v)\cap N(v_i) \subset H$ contains a $K_{r-2}$ not using the vertices of $V_i$.
\end{proof}

\section{Concluding remarks}
\label{sec:conclude}

We have reduced the problem of determining $sat(n,k,r)$ for large $n$ to that of $\alpha(k,r)$. Although, we have determined $\alpha(k,r)$ for some values of $k$ and $r$, a large number of cases remain unknown. In particular, the seemingly easiest case when $r$ is fixed and $k$ is large,  is still open.
\begin{problem}
Determine $\alpha(k,r)$ for $k \ge 2r-2$ and $r \equiv 1,3\mod{6}$. 
\end{problem}

For $k \ge 2r-2$ and $r \equiv 0,2,4,5\mod{6}$, we have determined $\alpha(k,r)$ except one missing case when $3$ is the smallest divisor of $r-2$ and $k = 2r-2$. Theorem~\ref{thm:alpha}$(i)$ implies that $\alpha(2r-2,r) \in \{(2r-3)^2,(2r-3)^2-1\}$ and we suspect that $\alpha(2r-2,r) = (2r-3)^2$. 

Not only we believe that $\beta_2(k,r) = 4r-k-2$ for $r < k \le 2r-1$ (see Conjecture~\ref{conj:beta2}) but we also think that the upper bound $\alpha(k,r) \le (k-1)\beta_2(k,r-1) \le (k-1)(4r-k-6)$ in Theorem~\ref{thm:alpha}$(i)$ is the correct value for $\alpha(k,r)$ in this case.
\begin{conjecture}
$\alpha(k,r) = (k-1)(4r-k-6)$ for $5 \le r \le k \le 2r-4$.
\end{conjecture}

We have shown that $33 \le \alpha(5, 5) \le 36$.
This is the smallest case for which the value of $\alpha$ is not yet known.
\begin{problem}
    Find $\alpha(5, 5)$.
\end{problem}
To prove the lower and upper bounds on $\alpha(k, r)$, we extensively used the bounds on $\beta_{1}(k, r)$ and $\beta_{2}(k, r)$.
We believe that determining the values of $\beta_{i}(k, r)$ is an interesting problem on its own.
\begin{problem}
    Determine $\beta_{i}(k, r)$ for  $k \ge r \ge 2$ and $2 \le i \le k - r +1$.
\end{problem}

We end the paper with a remark on a related problem. Recall that $sat(n,K_r)$ is the minimum number of edges in a $K_r$-free graph on $n$ vertices but the addition of an edge joining any two non-adjacent vertices creates a $K_r$. In the pioneer paper of Erd{\H o}s, Hajnal, and Moon~\citep{Erdos}, they determined $sat(n,K_r)$ by considering a more general problem where the graphs were not required to be $K_r$-free. Interestingly, the two problems have the same answer since the extremal graph is $K_r$-free. We remark that this phenomenon does not happen for partite saturation. Roberts~\citep{Roberts} studied the corresponding more general problem for $sat(K_{r \times n},K_r)$ and showed that the minimum number of edges in a $K_r$-saturated subgraph of $K_{r \times n}$ where the subgraph is allowed to contain $K_r$ is $\binom{r}{2}(2n-1)$ for $r \ge 4$ and sufficiently large $n$. On the other hand, Theorem~\ref{thm:sat} and Theorem~\ref{thm:alpha} imply that $sat(K_{r \times n},K_r) \ge r(2r-4)n + o(n) > \binom{r}{2}(2n-1)$ for sufficiently large $n$.

\bibliographystyle{siam}
\bibliography{partsat-2}

\end{document}